\documentclass[reqno,12pt]{amsart}
\usepackage{amssymb,amsmath,amsthm}
\usepackage[usenames]{color}
\usepackage{multirow}
\newcommand{\leg}[2] {\left(\frac{#1}{#2}\right)}

\def\C{\mathbb{C}}
\def\Z{\mathbb{Z}}

\def\R{\mathbb{R}}

\def\H{\mathbb{H}}

\def\SL{{\rm SL}}

\renewcommand{\(}{\left(}
\renewcommand{\)}{\right)}

\def\la{\lambda}
\def\lap{\lambda^{\prime}}

\newtheorem{theorem}{Theorem}
\newtheorem{lemma}[theorem]{Lemma}
\newtheorem{corollary}[theorem]{Corollary}
\newtheorem{proposition}[theorem]{Proposition}

\theoremstyle{remark}
\newtheorem*{remark}{Remark}

\newtheorem{example}{Example}

\numberwithin{equation}{section}

\begin{document}


\title[Mock theta functions modulo  $3$]{Mock theta functions and weakly holomorphic modular forms  modulo $2$ and $3$}

\date{\today}

\author{Scott Ahlgren}
\address{Department of Mathematics\\
University of Illinois\\
Urbana, IL 61801} 
\email{sahlgren@illinois.edu} 
 \author{Byungchan Kim}
\address{School of Liberal Arts \\ Seoul National University of Science and Technology \\ 172 Gongreung 2 dong, Nowongu, Seoul,139-743, Korea}
\email{bkim4@seoultech.ac.kr}



\begin{abstract} 
We prove that the coefficients of the mock theta functions
\[
f(q) =  \sum_{n=1}^{\infty} \frac{ q^{n^2}}{(1+q)^2 (1+q^2)^2 \cdots (1+q^n)^2 }
\]
and 
\[
\omega(q):=1+\sum_{n=1}^\infty \frac{q^{2n^2+2n}}{(1+q)^2(1+q^3)^2\dots (1+q^{2n+1})^2} \]
possess no linear congruences modulo $3$.  We prove similar results for the moduli $2$ and $3$ 
for a wide class of weakly holomorphic modular forms and discuss applications.  This extends work of Radu
on the behavior of the ordinary partition function modulo $2$ and $3$.
\end{abstract}

\thanks{The first author was  supported by a grant from the Simmons Foundation (\#208525 to Scott Ahlgren).
The second author was supported by the Basic Science Research Program through the National Research Foundation of Korea (NRF) funded by the Ministry of Education, Science and Technology (NRF2012-0003238), and the TJ Park Science Fellowship from the  POSCO TJ Park Foundation.}


\maketitle

\section{Introduction} 
The arithmetic properties of  Fourier coefficients of modular forms are involved in many areas of number theory, and they
have formed the topic of a vast amount of research.  A prototypical modular form is the Dedekind eta function
\[\eta(z): = q^{1/24} \prod_{n=1}^{\infty} (1-q^n),\qquad q:=e^{2\pi iz},\]
whose inverse generates the partition function (the prototypical function of additive number theory) via the well-known relation
\begin{equation}\label{etadef}
 \frac1{\eta (z)}= q^{-1/24}\sum_{n=0}^{\infty} p(n) q^n.
\end{equation}
Most of what is known about the arithmetic properties of $p(n)$ stems from the interpretation of \eqref{etadef}
as a weakly holomorphic modular form of weight $-1/2$. Although there are far too many results to mention individually,
we mention Ono's  paper \cite{Ono:partition}  and the subsequent work \cite{AhlgrenOno} which show that Ramanujan's famous congruence
\begin{equation}\label{ramcong5}
p(5n+4) \equiv 0 \pmod{5}
\end{equation}
has an analogue for any modulus $M$ coprime to $6$  (see also the recent work of Folsom, Kent and Ono \cite{FolsomKentOno}).    Treneer \cite{Treneer1}, \cite{Treneer2} has extended these results to cover any weakly holomorphic modular form.
In the other direction, the first author and Boylan \cite{AhlgrenBoylan} showed that there are no congruences as simple as \eqref{ramcong5} for primes $\geq 13$.

The situation has been less clear for the primes which divide $24$.  
Parkin and Shanks \cite{ParkinShanks}
conjectured that
 \[
 \# \left\{ n \leq x: p(n) \not\equiv 0 \pmod{2} \right\}  \sim \frac{x}{2},\]
 and there is an analogous folklore conjecture modulo $3$.
The current results are far from this expectation; for example 
the best result for the number of odd values of $p(n)$  is  due to Nicolas \cite{Nicolas}, who obtained the bound
 $\sqrt{x}(\log\log x)^K/\log x$ for any $K$.
An old conjecture of Subbarao \cite{Subbarao} states that there are no linear congruences for $p(n)$ modulo $2$.
In  striking recent work,   Radu \cite{Radu:Crelle} has proved Subbarao's conjecture, as well as its analog modulo $3$, with clever and  technical arguments.
In this paper we will adapt the methods of Radu's paper to prove non-vanishing results for the coefficients of mock theta functions
and weakly holomorphic modular forms of certain types.

The function 
\[
f(q) = \sum_{n=0}^{\infty} a(n)q^n=1 + \sum_{n=1}^{\infty} \frac{ q^{n^2}}{(1+q)^2 (1+q^2)^2 \cdots (1+q^n)^2 }
\]
is a prototypical mock theta function  (see, for example, works of   Zwegers \cite{Zwegers:2001},   Bringmann and   Ono \cite{Ono:fq}, and   Zagier \cite{Zagier:2009}).  Using a standard argument and  \cite{Treneer1}, \cite{Treneer2}, one deduces that there are linear congruences $a(mn+t)\equiv 0\pmod {\ell^j}$ for any prime power $\ell^j$ with $\ell\geq 5$.
The function $f(q)$ coincides modulo $2$  with the generating function for partitions.  To see this, 
define the {\em rank} of a partition $\la$ as  $\la_1 - \ell (\la)$, where $\la_1$
is the largest part of $\la$ and $\ell(\la)$ is the number of parts. 
 Define  $N_{e}(n)$  and  $N_{o}(n)$ as the number of partitions of $n$ with even and odd ranks, respectively. Then we have 
\begin{equation}\label{fqrank}
f(q) := \sum_{n=0}^{\infty} a(n) q^n = 1 + \sum_{n=1}^{\infty} \( N_{e} (n) - N_{o} (n) \) q^n .
\end{equation}
Since $p(n) = N_{e} (n) + N_{o} (n)$, we have $a(n) \equiv p(n) \pmod{2}$ for all   $n$, and 
Radu's result  implies that there are no linear congruences for $a(n)$ modulo $2$.
Here we will prove
\begin{theorem}\label{mainthm1}
For any positive integer $m$ and any integer $t$ we have 
\[\sum a(mn+t)q^n\not \equiv 0\pmod 3 .\]
\end{theorem}

The mock theta function  
\[
\omega(q):=1+\sum_{n=1}^\infty \frac{q^{2n^2+2n}}{(1+q)^2(1+q^3)^2\dots (1+q^{2n+1})^2} = \sum_{n=0}^{\infty} c(n) q^n
\]
appears naturally with $f(q)$ as the component of a vector-valued mock modular form (see, for example,   \cite{Zwegers:2001}).
Andrews \cite{Andrews:Durfee} gives a partition theoretic interpretation for $c(n)$ as the number of partitions of $n+1$ into nonnegative integers such that every part in the partition, with the possible
exception of the largest part,  appears as a pair of consecutive integers.  For example, there are $6$ such partitions of $5$:
\[
5,\;\; 4+(1+0),\;\;3+(1+0)+(1+0),\;\;2+(2+1),\;\;2+(1+0)+(1+0)+(1+0),\;\;(1+0)+\cdots+(1+0).
\]
This function behaves quite differently modulo $2$.  In fact, Andrews \cite[Theorem 31]{Andrews:Durfee}  has shown that $c(n)$ is odd if and only if $n = 6j^2 + 4j$ for some integer $j$.  For the modulus $3$, we will prove
\begin{theorem}\label{mainthm2}
For any positive integer $m$ and any integer $t$ we have 
\[\sum c(mn+t)q^n\not \equiv 0\pmod 3. \]
\end{theorem}

It is natural to ask how these results extend to more general classes of modular forms.
In this direction, we   investigate the class 
\[
\mathcal S(B, k, N, \chi) := \left\{ \eta^{B} (z) F(z) : \text{$F(z) \in M_{k}^{!} (\Gamma_0 (N) , \chi )$} \right\},
\]
where $k$ is an integer or   half-integer and $M_{k}^{!} (\Gamma_0 (N) , \chi )$ is the space of weakly holomorphic modular forms of weight $k$ and level $N$ with  character $\chi$ (see Section 2 for definitions). If $f(z) \in \mathcal S(B, k, N, \chi)$, then we have
\[
f(z) = q^{B/24} \sum_{n \ge n_0} a_{f} (n) q^n.
\]
We show that certain forms of this type do not possess  linear congruences modulo $2$ or $3$.
If $m$  is a positive integer and  $B$ is an  integer with $6 \nmid B$, then write $m=2^r 3^s m'$ with $(m', 6)=1$,
and define a divisor of $m$ by 
\begin{equation}\label{qmb}
Q_{m, B}=\begin{cases} 
m'\ \ &\text{if}\ \ (B, 6)=1,\\
2^rm'\ \ &\text{if}\ \ (B, 6)=2,\\
3^sm'\ \ &\text{if}\ \ (B, 6)=3.
\end{cases}
 \end{equation}

\begin{theorem} \label{etathm} Suppose that  $\ell=2$ or $\ell=3$.
Suppose that $f\in \mathcal S(B, k, N, \chi)$ and that $f$ has a pole at infinity and leading coefficient equal to $1$.
Suppose that      $\ell\nmid BN$ and that the coefficients of $f$ are $\ell$-integral rational numbers.
Then for any positive integers $m$ and $t$ with $(Q_{m, B},N)=1$, we have
\[
\sum a_{f} (m n + t)q^n \not\equiv 0 \pmod{\ell}.
\]
\end{theorem}

\begin{remark} The analogous statement will hold for forms  with algebraic coefficients, where $\ell$ is replaced by any prime ideal over $\ell$.
\end{remark}

As an application, we  consider eta-quotients, which we express in the standard form
\begin{equation}\label{etaproddef}
 f (z) = \prod_{\delta \mid N } \eta ( \delta z )^{r_{\delta}}.
\end{equation}
Writing 
\begin{equation}\label{bfdef}
B=B_f:=\sum \delta r_{\delta},
\end{equation}
we have 
\begin{equation}\label{etaprodcoeff}
f(z) = q^\frac {B}{24}\sum a_f (n) q^n.
\end{equation}
\begin{corollary} \label{etacor}
Suppose that $f(z)$ is an eta-quotient as in \eqref{etaproddef}--\eqref{etaprodcoeff} and that $f$ has a pole at infinity.
Suppose that  $\ell=2$ or $\ell=3$ and that $\ell\nmid B$.
Write $N=\ell^s N'$ with $\ell\nmid N'$.
Then for any 
positive integers $m$ and $t$ with $(Q_{m, B}, N' )=1$, we have 
\[
\sum a_f (m n + t )q^n \not\equiv 0 \pmod{\ell}.
\]
\end{corollary}

\begin{remark} The hypotheses are   satisfied if $\ell\nmid BN$ and $(m, N)=1$.
\end{remark}

We  give some examples involving   Corollary \ref{etacor}.
\begin{example}
A $k$-multipartition of $n$ is a $k$-tuple of partitions $(\pi_1 , \pi_2, \ldots, \pi_k)$ such that  $| \pi_1 | + | \pi_2 | + \cdots + | \pi_k | =n$.
 The generating function for $k$-multipartitions is 
 \[
\sum_{n=0}^{\infty} p_{k} (n) q^{n - k/24} = \eta^{-k} (z).
\]
Various congruences for $p_k (n)$ have been  studied (see, for example, \cite{multi1}, \cite{multi2}, \cite{multi3}). For example,  Andrews \cite{multi1} showed that for each prime $p\geq 5$ there are $(p+1)/2$ values of $b$ with $1 \le b \le p$ for which $p_{p-3} (pn+b) \equiv 0 \pmod{p}$.  
Corollary \ref{etacor} shows that if $\ell=2$ or $\ell=3$ and  $\ell \nmid k$, then there are no linear congruences for 
  $p_{k} (n)$ modulo $\ell$.
\end{example}

\begin{example}
A cubic partition of $n$ is a bi-partition ($\pi_1$, $\pi_2$) such that $\pi_2$ contains no odd part. 
For example, the cubic partitions of $3$ are 
\[
(3, \varnothing ), \quad (2+1, \varnothing) , \quad (1+1+1, \varnothing), \quad (1,2).
\]
The generating function for these partitions is 
\[
f_{\operatorname{cu}}(z)=\sum_{n=0}^{\infty} \operatorname{cu}(n) q^{n-3/24} = \eta^{-1} (z) \eta^{-1} (2z).
\]
In this case, the quantity $B$ from \eqref{bfdef} is a multiple of $3$.  So Corollary~\ref{etacor} does not apply for the prime $\ell=3$.
In fact,   H.-C. Chan \cite{HeiChi} has shown that
\[
\operatorname{cu}(3n+2) \equiv 0 \pmod{3}.
\]
 Corollary \ref{etacor} implies that there is no linear congruence for $\operatorname{cu}(n)$ modulo $2$.
\end{example}

\begin{example}
To explain Ramanujan's     congruences for $p(n)$, the {\em crank} of a  partition was introduced by  Andrews and   Garvan \cite{AGcrank}. 
Let $M_{e} (n)$  and $M_{o} (n)$ denote the  number of partitions of $n$ with even and odd  crank, respectively.   
As a counterpart to \eqref{fqrank}, we have
\[
\sum_{n=0}^{\infty} (M_{e} (n) - M_{o} (n) ) q^{n - 1/24} = \frac{\eta^{3}(z)}{\eta^{2}(2z)}.
\]
In \cite{CLK},  Choi, Lovejoy and  Kang showed that
\[
M_{e} (5n+4) - M_{o} (5n+4) \equiv 0 \pmod{5}.
\]
Here $B=-1$, so 
Corollary \ref{etacor} guarantees that there are no  linear congruences modulo $2$ or $3$.  
\end{example}

\begin{example}
Andrews  \cite{Frob1} introduced the  generalized Frobenius symbol $c\phi_2$ and showed that
\[
\sum_{n=0}^{\infty} c\phi_2 (n) q^{n-1/12} = \frac{\eta^5(2z)}{\eta^4 (z)\eta^2(4z)}.
\]
Andrews  \cite[Cor 10.1 and Thm 10.2]{Frob1} proved that
\[
c\phi_2 ( 2n+1) \equiv 0 \pmod{2} \quad \text{and} \quad c\phi_2 (5n+3) \pmod{5},
\]
and many congruence properties of these symbols have since been investigated  (see,  for example,  \cite{Frob2}, \cite{PauleRadu}, and \cite{Frob3}). 
 Corollary \ref{etacor} shows that there is no linear congruence of the form $c\phi_2 (mn+t)$ modulo $3$ with odd $m$  (it is likely that an adaptation of these methods can
 be used to remove the restriction on $m$ in this case).

 \end{example}

\begin{example}
The assumption that $f$ has a pole at  infinity is necessary.  The function $\eta(z)=q^\frac1{24}\sum(-1)^k q^\frac{k(3k+1)}2$ provides a simple example.
 For another example, the generating function for the number of $4$-core partitions of $n$ is given by
\[
\sum_{n=0}^{\infty} c_{4} (n) q^{n+15/24} = \frac{\eta^{4} (4z)}{\eta(z)},
\]
and  M. Hirschhorn and J. Sellers \cite{HS4core} have shown that
\[
c_{4} (9n+2) \equiv 0 \pmod{2}.
\]
\end{example}
\begin{example}  The assumption that $(Q_{m, B}, N)=1$ is also necessary in general.
For example, if we define $a(n)$ by
\[
\sum_{n=0}^{\infty} a(n) q^{n-5/24} = \eta^{-1} (5z),
\]
then we have $a(5n+1) \equiv a(5n+2) \equiv \cdots \equiv a(5n+4) \equiv 0 \pmod{2}$. 
\end{example}

  The 
 first author and   Boylan \cite{ABparity} proved that if $\ell$ is prime and $f=\sum a(n)q^n$ is a weakly holomorphic modular form
 with $f\not\equiv 0\pmod \ell$, then 
 then 
\begin{equation}\label{ABresult}
\# \left\{ n \leq x : a_f (n) \not\equiv 0 \pmod{ \ell } \right\}  \gg_{f,K} \frac{\sqrt{x}}{\log x} (\log\log x)^{K}
\end{equation}
for any positive integer $K$.
In each situation where the results described above imply  that $\sum a(mn+t)q^n\not\equiv 0\pmod\ell$, 
the lower bound \eqref{ABresult} applies to the number of non-zero coefficients (this will be clear from the method of proof).

The proofs follow the outline given by  Radu \cite{Radu:Crelle}, and require a careful analysis of the integrality properties of the 
functions $\sum a(mn+t)q^n$ at various cusps.  In Section 2, we give  some background on modular forms and mock modular forms. In Sections 3 and 4, we prove Theorems \ref{mainthm1} and \ref{mainthm2}. In Section 5, we prove Theorem \ref{etathm} and its corollary.

\section{Preliminaries}
We recall the definitions of harmonic weak Maass forms and mock modular forms (see, for example, \cite{Bruinier:2004fk}, \cite{Ono:2009} or \cite{Zagier:2009} for details). 
Given  $k\in \frac{1}{2}\Z\setminus \Z$, $z=x+iy$ with $x, y\in \R$, $4\mid N$ and   a Dirichlet character  $\chi$ modulo $N$,  a {\it harmonic weak Maass form of weight $k$ with Nebentypus $\chi$ on $\Gamma_0(N)$} is a  smooth function $F:\H\to \C$ satisfying the following:
\begin{enumerate}
\item For all $ \left(\begin{smallmatrix}a&b\\c&d
\end{smallmatrix} \right)\in \Gamma_0(N)$ and all $z \in \H$, we
have
\begin{displaymath}
F \left(\frac{a z +b}{c z +d} \right)
= \leg{c}{d}\epsilon_d^{-2k} \chi(d)\,(cz+d)^{k}\ F(z),
\end{displaymath}
where \begin{equation*}
\epsilon_d:=\begin{cases} 1 \ \ \ \ &{\text {\rm if}}\ d\equiv
1\pmod 4,\\
i \ \ \ \ &{\text {\rm if}}\ d\equiv 3\pmod 4. \end{cases}
\end{equation*}
\item  $\Delta_k(F)=0$, where $\Delta_k$ is the  weight $k$ hyperbolic Laplacian, given by
\begin{equation*}\label{laplacian}
\Delta_k := -y^2\left( \frac{\partial^2}{\partial x^2} +
\frac{\partial^2}{\partial y^2}\right) + iky\left(
\frac{\partial}{\partial x}+i \frac{\partial}{\partial y}\right).
\end{equation*}
\item The function $F$ has
at most linear exponential growth at the cusps of $\Gamma_0(N)$.
\end{enumerate}
We denote by $H_{k}\left(\Gamma_0(N), \chi\right)$ the space of harmonic weak Maass forms of weight $k$ with Nebentypus $\chi$ on  $\Gamma_0(N)$.
We denote the  subspace of weakly holomorphic forms (i.e., those meromorphic forms whose poles are supported at the cusps of $\Gamma_0(N)$) by  $M_{k}^!\left(\Gamma_0(N), \chi\right)$, and the space of holomorphic forms by $M_{k}\left(\Gamma_0(N), \chi\right)$   (if $\chi$ is trivial, then we drop it from the notation).
Each harmonic weak Maass form $F$ decomposes uniquely  as the sum of  a holomorphic part $F^{+}$ and a non-holomorphic part $F^{-}$.
The holomorphic part, which is  known as a \textit{mock modular form},  is a power series in $q$ with at most finitely many negative exponents.

Now define 
\begin{equation}\label{FDEF}
F(z)=(F_0(z), F_1(z), F_2(z))^T:=\left(q^{-1/24}f(q),\  2q^{1/3}\omega(q^{1/2}), \ 2q^{1/3}\omega(-q^{1/2})\right)^T
\end{equation}
and 
 \[G:=(g_1, g_0, -g_2)^T,\]
 where the $g_i (z)$ are theta functions defined by 
\begin{equation}\begin{aligned}\label{TF}
g_0(z)&=\sum_{n\in\Z}(-1)^n\left(n+\tfrac13\right)q^{\frac32\left(n+\frac13\right)^2},\\
g_1(z)&=-\sum_{n\in\Z} \left(n+\tfrac16\right)q^{\frac32\left(n+\frac16\right)^2},\\
g_2(z)&=\sum_{n\in\Z} \left(n+\tfrac13\right)q^{\frac32\left(n+\frac13\right)^2}.\\
\end{aligned}
\end{equation}
 Zwegers \cite{Zwegers:2001} showed that 
 \begin{equation}\label{zwegersdef}
 T(z):=F(z)-2i\sqrt{3}\int_{-\overline z}^{i\infty}\frac{G(\tau)}{(-i(z+\tau))^\frac12}\, d\tau
 \end{equation}
 transforms as 
\begin{equation} \label{htrans}\begin{aligned}
T(z+1)=&\begin{pmatrix} \zeta_{24}^{-1}& 0&0\cr
0&0&\zeta_3\cr
0&\zeta_3&0\cr
\end{pmatrix}T(z),\\
T(-1/z)=\sqrt{-i z}&\begin{pmatrix} 0& 1&0\cr
1&0&0\cr
0&0&-1\cr
\end{pmatrix}T(z).\\
\end{aligned}
\end{equation}
We also require the  incomplete gamma function, given by 
\[\Gamma(\alpha, x):=\int_x^\infty e^{-t}t^{\alpha-1}\, dt.\]

\section{Proof of Theorem \ref{mainthm1}}
We work with the function 
\[M(z)=F_0(z)-2i\sqrt{3}\int_{-\overline z}^{i\infty}\frac{g_1(\tau)}{(-i(z+\tau))^\frac12}\, d\tau. \]
A computation using \eqref{FDEF}, \eqref{TF} and \eqref{zwegersdef} shows that  $M(z)$ has the form
\begin{equation*}M(z)=q^{-1/24}\sum a(n)q^n+q^{-1/24}\sum b(n) \Gamma\left(\tfrac12, 4\pi y (n+\tfrac1{24})\right)q^{-n},\end{equation*}
where
\begin{equation}
\label{quadcond}
b(n)=0 \qquad \text{unless}\qquad n=\frac{k(3k+1)}2\qquad \text {for some integer $k$}.
\end{equation}

For ease of notation we will write 
\begin{equation}M(z)=H(z)+NH(z),\label{Mform}
\end{equation}
where 
\begin{equation}H(z)=F_0(z)=q^{-1/24}\sum a(n)q^n,\label{Hform}
\end{equation}
and $NH(z)$ is the non-holomorphic part.

Suppose that $m$ is a positive integer and that $t$ is a non-negative integer.
Setting $\zeta_m:=e^{\frac{2\pi i}m}$, we define
\begin{equation}\label{mmtdef}
M_{m,t} := \frac1m\sum_{\la=0}^{m-1}M_{\la, m, t}
\end{equation}
where 
\begin{equation}\label{mlmtdef}
M_{\la, m, t} := \zeta_{m}^{-\la (t-1/24)} M\( \( \begin{matrix} 1 & \la \\ 0 & m \end{matrix} \) z\).
\end{equation}
Using notation in analogy with \eqref{Mform}, we  write
\[M_{m,t}=H_{m,t}+NH_{m,t}.\]
A calculation gives
\begin{equation}\label{hmt}
H_{m,t} = q^{ \frac{t-1/24}{m}} \sum a(mn+t) q^n
\end{equation}
and 
\begin{equation}\label{gmt}
NH_{m,t} = q^{ \frac{t-1/24}{m}} \sum   b(mn-t) \Gamma\left(\tfrac12, 4\pi y\cdot(n+\tfrac{1/24-t}m  )  \right)q^{-n}.
\end{equation}

In order to prove Theorem \ref{mainthm1} we must show that for any progression $t\pmod m$ we have 
\begin{equation}\label{goal}
H_{m, t}\not \equiv 0\pmod 3.
\end{equation}
We call a progression $t\pmod m$ {\it good} if for some $p\mid m$ we have 
$\leg{1-24t}p=-1$.  By \eqref{quadcond} and \eqref{gmt} we see that  if $t\pmod m$ is good, then  $NH_{m, t}=0$.
Suppose that the progression $ t\pmod m$ is not good.  In this case we  pick a prime $p\geq 5$ with $p\nmid m$ and a quadratic non-residue $x\pmod p$,
and we find a  solution to the system of congruences
\begin{align*}
T&\equiv t\pmod m\\
T&\equiv \frac{1-x}{24}\pmod p.
\end{align*}
After replacing $t\pmod m$ by the sub-progression $T\pmod{mp}$,
we are reduced in proving  \eqref{goal} to considering progressions which are good. 

The next two propositions describe the transformation properties of the forms $M_{m, t}$.
Given a positive integer $m$, we define
\begin{equation}\label{nmdef}
N_m:=\begin{cases} 
2m\ \ \ &\text{if}\ \ \ (m, 6)=1,\\
8m\ \ \ &\text{if}\ \ \ (m, 6)=2,\\
6m\ \ \ &\text{if}\ \ \ (m, 6)=3,\\
24m\ \ \ &\text{if}\ \ \ (m, 6)=6.\\
\end{cases}
\end{equation}

\begin{proposition} \label{modular}
Suppose that $t\pmod m$ is good.
Then  
\[M_{m, t}^{24m}=H_{m, t}^{24m} \in M_{12m}^!(\Gamma_1(N_m)).\]
\end{proposition}
Now if $A=\begin{pmatrix} a & b \\ c & d \end{pmatrix}\in \Gamma_0(N_m)$ has $3\nmid a$ then
we define 
$t_A$ to be any integer with 
\begin{equation}\label{tadef}
t_A  \equiv a^2 t + \frac{1-a^2}{24} \pmod{m}.
\end{equation}

\begin{proposition} \label{Mtransprop} 
Suppose that $t\pmod m$ is good.
Then for every $A=\begin{pmatrix} a & b \\ c & d \end{pmatrix}\in \Gamma_0(N_m)$ with $3\nmid a$ we have
\begin{equation}\label{gamma0trans1}
 M_{m, t}^{24m}\big|_{12m}A
=  M_{m, t_A}^{24m}.
\end{equation}
\end{proposition}

In each case, the matrices $A$ as above with $3\nmid a$ generate $\Gamma_1(N_m)$.
Therefore  Proposition~\ref{modular}  follows immediately from Proposition~\ref{Mtransprop}.
For Proposition~\ref{Mtransprop} we require
a transformation law from the work of   Bringman and  Ono \cite[p. 251]{Ono:fq}.  
In particular,   for $A:=\begin{pmatrix} a & b \\ c & d \end{pmatrix}\in \Gamma_0(2)$ with $c>0$, we have
\begin{equation}\label{Mtrans}
M\left(\frac{az+b}{cz+d}\right)=w(A)\cdot (cz+d)^\frac12 M(z),
\end{equation}
where 
$w(A)$ is the root of unity given by
\begin{equation}\label{wabcd}
w(A):=i^{-\frac12}\cdot e^{-\pi i s(-d, c)}\cdot (-1)^\frac{c+1+ad}2\cdot
e^{2\pi i(-\frac{a+d}{24c}-\frac a 4+\frac{3 d c}8)}
\end{equation}
and   $s(d,c)$ is the Dedekind sum defined by 
\begin{equation}\label{dedsum}
s(d,c) = \sum_{r=1}^{c-1} \left( \frac{r}{c} -  \left\lfloor \frac{r}{c} \right\rfloor - \frac{1}{2} \right) \left( \frac{dr}{c} - \left\lfloor \frac{dr}{c} \right\rfloor - \frac{1}{2} \right).\end{equation}

For any $A\in \SL_2(\Z)$, we have the following (see, e.g., \cite[p. 247]{Lewis}):
\begin{equation}\label{lewis}
12s(-d, c)+\frac{a+d}c\in \Z,
\end{equation}
and so 
\begin{equation}\label{wa}
w(A)^{24}=1.\end{equation}
We also require a transformation law which follows from Lemma~2 of Lewis  \cite{Lewis} (we have corrected
a sign error in the statement of that lemma).
\begin{lemma}
\label{LewisLem2}
Let $m$ be a positive integer.  Then for every $A=\begin{pmatrix} a & b \\ c & d \end{pmatrix}\in \Gamma_0(N_m)$ with   $c>0$ and $3 \nmid a$ we have
\[
s ( d+c , mc ) = s (d,mc) + \frac{ 1-a^2 }{12m} + \text{even integer}.
\]
\end{lemma}

\begin{proof}[Proof of Proposition~\ref{Mtransprop}]

Suppose that   $A=\begin{pmatrix} a & b \\ c & d \end{pmatrix}\in \Gamma_0(N_m)$ has $3\nmid a$.
We may suppose that $c>0$.
For each $\la\pmod m$, \eqref{mlmtdef} gives
\begin{equation}\label{trans1}
M_{\la, m, t}( Az) =  \zeta_{m}^{- \la (t -1/24)} M\( \begin{pmatrix} 1 & \la \\ 0 & m \end{pmatrix}  Az\)  
=  \zeta_{m}^{-\la ( t -1/24 )} M \(A_\la\begin{pmatrix} 1 & \la^{\prime} \\ 0 & m \end{pmatrix}  z  \),
\end{equation}
where
\begin{equation}\label{alamdef}
A_\la = \( \begin{matrix} a+ c \la & \frac{-\la^{\prime} c \la - \la^{\prime} a + b + d \la }{m} \\ mc & d - c \la^{\prime} \end{matrix} \),
\end{equation}
and  $\la^{\prime} \in \{ 0,1,\ldots,m-1 \}$ is chosen with 
\begin{equation}\label{lamprimedef}
a \la^{\prime} \equiv b + d \la \pmod{m}.
\end{equation}
Note that $\la'$ travels the residue classes mod $m$ with $\la$.

Using \eqref{trans1} and recalling the definitions \eqref{mlmtdef} and  \eqref{tadef}, we obtain 
\begin{equation}\label{maz}
M_{\la, m, t}( Az) =(cz+d)^\frac12 w(A_\la)
\zeta_{m}^{-\la ( t  -1/24 )}\zeta_{m}^{\la' (t_A  -1/24 )}M_{\la', m, t_A}(z).
\end{equation}
We find that
\begin{align*}
w(A_\la)
\zeta_{m}&^{-\la ( t  -1/24 )}\zeta_{m}^{\la' (t_A  -1/24 )}\\
&=
i^{-\frac12}e^{-\pi i s(-d+c\lambda', mc)} \cdot (-1)^\frac{mc+1+(a+c\la)(d-c\lap)}2  
\cdot e^{2\pi i(-\frac{a+d}{24mc}-\frac{a +c\la} 4+\frac{3mc(d-c\lap)}8)}\zeta_{m}^{-\la  t+\la't_A}\\
&=\zeta_1(A, m)e^{-\pi i s(-d+c\lambda', mc)}\cdot (-1)^\frac{-a c\lap+cd\la}2 
\cdot e^{2\pi i(-\frac{a+d}{24mc}-\frac{c\la} 4-\frac{3mc^2\lap}8)}\zeta_{m}^{-\la  t+\la't_A},
\end{align*}
where $\zeta_1(A,m)$ is a $24$th root of unity which depends only on $A$ and $m$ (and not on $\la$).
We see that     
\begin{equation}\label{minus1}
  (-1)^\frac{-a c\lap+cd\la}2 
\cdot e^{2\pi i(-\frac{ c\la} 4-\frac{3mc^2\lap}8)}=1
\end{equation}
by writing $c=2c_0$ and separating cases according to the parity of $c_0$
(note that if $c_0$ is odd then $m$ is odd).
Therefore we have
\[w(A_\la)
\zeta_{m}^{-\la ( t  -1/24 )} \zeta_{m}^{\la' (t_A  -1/24 )}
=\zeta_1(A, m)e^{-\pi i s(-d+c\lambda', mc)} 
\cdot e^{2\pi i(-\frac{a+d}{24mc})}\zeta_{m}^{-\la  t+\la't_A}.
\]
We apply Lemma \ref{LewisLem2} iteratively to the matrices 
\[\begin{pmatrix}  -a& b-ja\cr c&-d+jc\cr
\end{pmatrix}\in \Gamma_0(N_m),\ \ \ 1\leq j\leq \lap-1\]
to find that
\[
s ( -d + \la^{\prime} c , mc) = s (-d, mc ) +\la^{\prime} \frac{1-a^2}{12m} + \text{an even integer} .
\]
Also, since  $\begin{pmatrix} a(1-bc) & -b^2c/m\cr mc & d\end{pmatrix}\in \SL_2(\Z)$, 
we see from \eqref{lewis} that 
\[12s(-d, mc)+\frac{a+d}{mc}-\frac{ab}m\in \Z.\]
It follows from the last three equations that
\begin{equation}\label{simp}
w(A_\la)
\zeta_{m}^{-\la ( t  -1/24 )} \zeta_{m}^{\la' (t_A  -1/24 )}
 =\zeta_2(A, m) e^{-2\pi i\lap\frac{1-a^2}{24m}}\zeta_{m}^{-\la  t+\la't_A},
\end{equation}
where $\zeta_2(A,m)$ is a $24m^{\rm th}$ root of unity which depends only on $A$ and $m$.

Recalling \eqref{tadef} and the fact that
$\la   \equiv a^2\la'-ab  \pmod{m}$, we find that
\begin{equation}\label{simp2}
\zeta_{m}^{-\la  t+\la't_A}= \zeta_m^{abt} \zeta_m^{\lambda'\frac{1-a^2}{24}}.
\end{equation}
Combining \eqref{simp}  and \eqref{simp2}, we conclude that
\begin{equation}\label{simp3}
w(A_\la)
\zeta_{m}^{-\la ( t  -1/24 )} \zeta_{m}^{\la' (t_A  -1/24 )}
 =\zeta_3(A, m), 
\end{equation}
where $\zeta_3(A,m)$ is a $24m^{\rm th}$ root of unity which depends only on $A$ and $m$.
 Proposition~\ref{Mtransprop} follows immediately from \eqref{simp3},  \eqref{maz}, and \eqref{mmtdef}.
\end{proof}

\begin{lemma} \label{reductionlemma}
Suppose that $t \pmod {m}$ is good.    Write $m=2^s3^r Q$ with $(Q, 6)=1$.  
If \[\sum a(mn+t) q^n\equiv 0\pmod 3\] then   \[\sum a(Qn+t) q^n\equiv 0\pmod 3.\]
\end{lemma}

\begin{proof}   

The arithmetic progression $t\pmod Q$ is the disjoint union of the arithmetic progressions
\begin{equation}\label{arithprog}
t+\ell Q\pmod m, \ \ \ \ 0\leq \ell< 2^s 3^r.
\end{equation}
By the argument in \cite[Theorem 4.2]{Radu:Crelle},
we see that as $a$ ranges over integers with $(a, 6m)=1$, the quantity 
\[
t_{A} \equiv ta^2 + \frac{1-a^2}{24} \pmod{m}
\]
covers each of the progressions in \eqref{arithprog}.

Let $\Delta$ be the usual normalized cusp form of level one and weight 12.
By Proposition \ref{modular} there is a positive integer $j$ such that
\[
M_{m,t}^{24m} \Delta^{j} \in M_{12(m + j)} ( \Gamma_{1} (N_m) ).
\]
If $A=\begin{pmatrix} a & b \\ c & d \end{pmatrix}\in \Gamma_0(N_m)$ has $3\nmid a$ then 
by 
Proposition \ref{Mtransprop} we have
\begin{equation}\label{orbits}
M_{m,t}^{24m} \Delta^{j} \Big|_{12(m+j)} A = M_{m,t_{A}}^{24m} \Delta^{j}.
\end{equation}

We require a   fact proved by Deligne and Rapoport (see \cite[VII, Corollary 3.12]{Deligne} or  \cite[Corollary 5.3]{Radu:Crelle}):
If $f\in M_k(\Gamma_1(N))$ has coefficients in $\Z[\zeta_N]$ then the same is true of $f|_k\gamma$ for each $\gamma\in\Gamma_0(N)$.

It follows from \eqref{orbits} that if  $M_{m, t}\equiv 0\pmod 3$, then for each $t_A$ with $(a, 6m)=1$ we have 
$M_{m, t_A}\equiv 0\pmod 3$.
By \eqref{arithprog} we conclude that $M_{Q, t}\equiv 0 \pmod 3$, as desired.
\end{proof}

After Lemma \ref{reductionlemma} we are reduced to proving that if $Q$ is a positive integer with $(Q, 6)=1$ and $t\pmod Q$ is good,
then 
\begin{equation}\label{etp}
M_{Q, t}\not\equiv 0\pmod 3.
\end{equation}
 By Proposition~\ref{modular} we see that for sufficiently large $j$,  we have 
\begin{equation}\label{fsQt}
h:= M_{Q,t}^{24Q} \Delta^{j} \in M_{12(Q+j)} (\Gamma_{1} (2Q ) ).
\end{equation}
We compute the first term in the expansion of $h$ at the cusp $1/2$.
\begin{proposition}\label{cusp12}
Let $h$ be the modular form   defined in \eqref{fsQt}. Then we have
\[
h \Big|_{12(Q+j)}  \( \begin{matrix} 1 & 0 \\ 2 & 1 \end{matrix} \) = Q^{-12Q}q^{-Q^2 + j } + \cdots.
\]
\end{proposition}

\begin{proof}
By \eqref{mmtdef} and \eqref{mlmtdef} we have
\begin{align}\label{mqt}
M_{Q,t} \( \( \begin{matrix} 1 & 0 \\ 2 & 1 \end{matrix} \) z \) = \frac{1}{Q} \sum_{\la=0}^{Q-1} \zeta_{Q}^{ - \la ( t-1/24)}  M\( \( \begin{matrix} 1 & \la \\ 0 & Q \end{matrix} \)  \( \begin{matrix} 1 & 0 \\ 2 & 1 \end{matrix} \) z \).
\end{align}
We find that
\[
  \begin{pmatrix} 1 & \la \\ 0 & Q \end{pmatrix} \begin{pmatrix} 1 & 0 \\ 2 & 1 \end{pmatrix}  \\
= C_\la \begin{pmatrix} 1 & \la^{\prime} \\ 0 & Q/d_{\la} \end{pmatrix} \begin{pmatrix} d_{\la} & 0 \\ 0 & 1 \end{pmatrix} ,
\]
where 
\[d_{\la}: = ( 1+2  \la, Q),\] 
$\la^{\prime} \in \{0,1,2,\ldots,\frac{Q}{d_{\la}} - 1 \}$ is uniquely defined by the congruence
 \[\frac{1+2\la}{d_{\la}} \la^{\prime} \equiv \la \pmod{  Q/d_{\la}},\] and
\[
C_\la :=\begin{pmatrix} \frac{1 + 2  \la}{d_{\la}} & \frac{ - \frac{1+2 \la}{d_{\la}} \la^{\prime} + \la}{Q/d_{\la}} \\ 2  Q/d_{\la} & -2  \la^{\prime} + d_{\la} \end{pmatrix}  \in \Gamma_{0} (2).
\]
By \eqref{Mtrans}, we obtain  
\begin{equation}\label{MC}
M \(C_\la  \frac{d_{\la} z + \la^{\prime}}{Q/d_{\la}} \) = w(C_\la)\cdot \( d_{\la} ( 2 z+1) \)^{1/2}  M  \( \frac{d_{\la} z + \la^{\prime}}{Q/d_{\la}} \).
\end{equation}
Since $M=q^{-1/24}+\dots$, we see from \eqref{mqt} and \eqref{MC} that the leading term in the expansion
of 
\begin{equation}\label{mqtexp}
(2z+1)^{-1/2}M_{Q,t}   \( \begin{pmatrix} 1 & 0 \\ 2 & 1 \end{pmatrix} z \)
\end{equation}
arises from those   $\lambda$ with $d_\lambda=Q$.
The only such $\la $ is $\la_0=(Q-1)/2$, in which case we have $\la_0'=0$ and 
$C_{\la_0} :=\begin{pmatrix}1& (Q-1)/2 \\ 2 & Q \end{pmatrix}.$
Therefore the  leading term in \eqref{mqtexp} is 
\[\frac1{\sqrt{Q}}  w(C_{\la_0}) \zeta_Q^{\left(\frac{1-Q}2\right)(t-1/24)}\cdot q^{\frac{-Q}{24}}+\dots.\]
  Proposition~\ref{cusp12} follows immediately from \eqref{fsQt} and \eqref{wa}.
  \end{proof}

Theorem~\ref{mainthm1} now follows quickly.
\begin{proof}[Proof of Theorem \ref{mainthm1}]
  Deligne and Rapoport (\cite[Corollary 3.13 and \S 4.8]{Deligne} or  \cite[\S 12.3.5]{Diamond}) proved that 
if $f\in M_k(\Gamma_1(N))$ has coefficients in $\Z[\frac1N, \zeta_N]$, then the same is true of its expansion at each cusp.

Suppose by way of contradiction that \eqref{etp} is false.
Then the modular form $3^{-24Q} h\in M_{12(Q+j)}(\Gamma_1(2 Q))$ has integral coefficients.
It follows that the coefficients of 
\[3^{-24Q} h \Big|_{12(Q+j)}  \( \begin{matrix} 1 & 0 \\ 2 & 1 \end{matrix} \)\]
lie in $\Z[ \frac{1}{2Q}, \zeta_{2Q} ]$.   This contradicts Proposition~\ref{cusp12}, and Theorem~\ref{mainthm1} is proved.
\end{proof}

\section{Proof of  Theorem \ref{mainthm2}}
The flow of the argument is similar to that of the last section. We now work with the function
\[
\Omega (z) = 2  q^{2/3} \omega (q) -2i\sqrt{3}\int_{-\overline{2z}}^{i\infty}\frac{g_0 (\tau)}{(-i(2z+\tau))^\frac12}\, d\tau.\]
A computation shows that 
\[
\Omega(z)= 2 q^{2/3} \omega( q) + q^{-1/3} \sum d(n) \Gamma \( \tfrac{1}{2} , 4\pi (n +\tfrac{1}{3} ) y \) q^{-n} ,
\]
where 
\[
d(n)=0 \qquad \text{unless}\qquad  n = 3k^2 + 2k\qquad \text {for some integer $k$}.
\]
 To isolate   arithmetic progressions, we define
\[
\Omega_{m,t} := \frac{1}{m} \sum_{\la=0}^{m -1} \zeta_{m}^{-\la( t+2/3)} \Omega \( \( \begin{matrix} 1 & \la \\ 0 & m \end{matrix} \) z \).
\]
A calculation gives 
\[
\Omega_{m,t} = q^{\frac{t+2/3}{m}} \sum c(mn+t) q^n + q^{\frac{t+2/3}{m}} \sum d(mn-t-1) \Gamma \( \frac{1}{2}, 4\pi \( n - \frac{t+2/3}{m} \) y \) q^{-n}.
\]
In this case, we call the arithmetic progression $t \pmod{m}$  {\it good} if 
\[
\( \frac{-3t-2}{p} \) = -1 \quad\quad\text{for some prime $p \mid m$.}
\]
If $t\pmod m$ is good, then $\Omega_{m,t}$ is weakly holomorphic.
As in the last section, it suffices to prove that when $t\pmod m$ is good we have
\[\Omega_{m,t}\not\equiv 0\pmod 3.\]

The transformation properties of $\Omega(z)$ are described in work of Andrews  \cite[Theorems 2.1 and 2.4]{Andrews:trans}. 
Using these results with  \eqref{htrans}, we   find that for   $A=\( \begin{matrix} a & b \\ c & d \end{matrix} \) \in \rm{SL}_2 (\mathbb{Z})$ with $c>0$, we have 
\begin{equation}\label{OmegaTrans}
\Omega (Az) = \begin{cases}
w_1 (A) (cz +d)^{1/2} \Omega(z)  &\text{ if $c$ is even,}\\
w_2 (A)\cdot  2^{-1/2} (cz +d )^{1/2} M (z/2)  &\text{ if $d$ is even,}
\end{cases}
\end{equation}
where $w_1 (A)$ and $w_2(A)$ are the roots of unity defined by
\begin{equation}\label{omegadefine}\begin{aligned}
w_1 (A) &:= (-i)^{1/2} (-1)^{(a-1)/2}e^{-\pi i s(-d, c/2)} e^{2\pi i \( \frac{3ab}{4} - \frac{a+d}{12c} \)},\\
w_2 (A) &:= i^{1/2} (-1)^{(32a-d)/24c}   e^{-\pi i s(-d/2, c)} e^{ -\frac{\pi i}{2} ( 2a+b -3 -3ab+3a/c )}.
\end{aligned} 
\end{equation}
Note that we have fixed a sign error in \cite[Theorem 2.1]{Andrews:trans}.

\begin{proposition} \label{Omegatransprop} 
Suppose that $t\pmod m$ is good.
Let $N_m$ be as defined in $\eqref{nmdef}$.
For every $A=\begin{pmatrix} a & b \\ c & d \end{pmatrix}\in \Gamma_0(2N_m)$ with $3\nmid a$ we have
\begin{equation}\label{gamma0trans1omea}
 \Omega_{m, t}^{24m}\big|_{12m}A
=  \Omega_{m, t_A}^{24m},
\end{equation}
where \[t_A: = t a^2 + \tfrac23 (a^{2}-1).\]
\end{proposition}

\begin{proof}[Proof of Proposition \ref{Omegatransprop}]
Write 
\[\Omega_{\la,m,t} (z ) = \zeta_{m}^{- \la (t +2/3)} \Omega\( \begin{pmatrix} 1 & \la \\ 0 & m \end{pmatrix}  z\).\]
Suppose that $A \in \Gamma_0 (2N_m)$ with $3 \nmid a$. Then
\[
\Omega_{\la,m,t} (Az )  =  \zeta_{m}^{-\la(t+2/3)} \Omega \( A_\la \( \begin{matrix} 1 & \la^{\prime} \\ 0 & m \end{matrix} \) z \),
\]
where $A_\la$ and $\la^{\prime}$ are   defined as in \eqref{alamdef} and \eqref{lamprimedef}. Using \eqref{OmegaTrans}, we find that
\begin{equation}\label{step1}
\Omega_{\la,m,t} (Az ) =   (cz+d)^{\frac{1}{2}} w_1(A_\la) \zeta_m^{-\la(t+\frac23)}\zeta_m^{\la'(t_A+\frac23)}\Omega_{\lap, m, t_A} (z)
\end{equation}
Note that  $4m \mid c$ and that $ad\equiv 1\pmod{4m}$. Moreover, from   Lemma~\ref{LewisLem2}  and \eqref{lewis} we have
\[
s(-d+\lap c , mc/2) = s( -d, mc/2) + \lap \frac{1- a^2 }{6m} + \text{an even integer},
\]
and
\[
 s(-d, mc/2)+\frac{a+ d }{6mc}-\frac{ab}{6m} \in \mathbb{Z}/12.
\]
Therefore,
\begin{align*}
w_1(A_\la)&=(-i)^{1/2} (-1)^{ (a+c\la -1)/2}   e^{-\pi i s(-d+c \lap , mc/2)} 
e^{2\pi i\left(\frac {3 (a+c\la)(-\lap c \la - \lap a + b + d \la)}{4m}  - \frac{a+d}{12mc }+\frac{\lap-\la}{12m}\right)}\\
 &=\omega_1\cdot e^{-\pi i s(-d+c \lap , mc/2)}  e^{2\pi i \left(\frac{ 3 (-\lap a^2  +\la)} {4m}  - \frac{a+d}{12mc }+\frac{\lap-\la}{12m}\right)}\\
&= \omega_2\cdot   e^{2\pi i \left(\frac{ 3 (-\lap a^2  +\la)} {4m}+\lap\frac{a^2-1}{12m}+\frac{\lap-\la}{12m}\right)}
 \end{align*}
 where $\omega_1$, etc. denote $24m^{\rm th}$ roots of unity which depend only on $A$,  $m$, and $t$.

Using this together with the fact that $\la \equiv a^2 \lap - ab \pmod{m}$, we obtain
\begin{align*}
w_1(A_\la)\zeta_m^{-\la(t+\frac23)}\zeta_m^{\la'(t_A+\frac23)}&=
\omega_2\cdot\zeta_m^{-\la t+\lap(t_A-\frac23(a^2-1))}\\
&=  \omega_3\cdot\zeta_m^{\lap(t_A-a^2t-\frac23(a^2-1))}\\
&=\omega_3.
\end{align*}
Proposition~\ref{Omegatransprop} follows immediately from this together with \eqref{step1}.
\end{proof}

In this case we cannot pull out powers of $2$ from the arithmetic progressions in question.  We have
\begin{lemma} \label{reductionlemma2}
Suppose that $t \pmod {m}$ is good.    Write $m=3^r Q$ with $(Q, 3)=1$.  
If \[\sum c(mn+t) q^n\equiv 0\pmod 3\] then   \[\sum c(Qn+t) q^n\equiv 0\pmod 3.\]
\end{lemma}

\begin{proof}   
The arithmetic progression $t\pmod Q$ is the disjoint union of the arithmetic progressions
\begin{equation}\label{arithprog2}
t+\ell Q\pmod m, \ \ \ \ 0\leq \ell< 3^r.
\end{equation}
By the argument in \cite[Theorem 4.2]{Radu:Crelle},
we see that as $a$ ranges over integers with $3\nmid a$, the quantity 
\[
t_{A} \equiv ta^2 + \frac{2}{3} (1-a^2)  \pmod{m}
\]
covers each of the progressions in \eqref{arithprog2}.

By Proposition \ref{Omegatransprop} there is a positive integer $j$ such that
\[
\Omega_{m,t}^{24m} \Delta^{j} \in S_{12(m + j)} ( \Gamma_{1} (2N_m) ).
\]
If $A=\begin{pmatrix} a & b \\ c & d \end{pmatrix}\in \Gamma_0(2 N_m)$ has $3\nmid a$ then 
by 
Proposition \ref{Omegatransprop} we have
\[
\Omega_{m,t}^{24m} \Delta^{j} \Big|_{12(m+j)} A = \Omega_{m,t_{A}}^{24m} \Delta^{j}.
\]
If  $\Omega_{m, t}\equiv 0\pmod 3$, then the fact recorded after \eqref{orbits} shows that for each $t_A$ with $3\nmid a$ we have 
$\Omega_{m, t_A}\equiv 0\pmod 3$. By \eqref{arithprog2} we conclude that $\Omega_{Q, t}\equiv 0 \pmod 3$, as desired.
\end{proof}

 By Proposition~\ref{Omegatransprop} we see that for sufficiently large $j$,  we have 
\begin{equation}\label{omegasQt}
h_{\Omega}:= \Omega_{Q,t}^{24Q} \Delta^{j} \in M_{12(Q+j)} (\Gamma_{1} (4Q ) ).
\end{equation}
In this case, we compute the first term in the expansion of $h_{\Omega} \Big| \( \begin{matrix} 1 & 0 \\ 1& 1 \end{matrix}\)$.
\begin{proposition}\label{cusp13}
Let $h_{\Omega}$ be the modular form   defined in \eqref{omegasQt}. Then we have
\[
h_{\Omega} \Big|_{12(Q+j)}  \( \begin{matrix} 1 & 0 \\ 1 & 1 \end{matrix} \) = (-1)^Q\cdot(2Q)^{-12Q} q^{-\frac{Q^2} 2 + j } + \cdots.
\]
\end{proposition}

\begin{proof}
Recall that 
\begin{align}\label{omegaqt}
\Omega_{Q,t} \( \( \begin{matrix} 1 & 0 \\ 1 & 1 \end{matrix} \) z \) = \frac{1}{Q} \sum_{\la=0}^{Q-1} \zeta_{Q}^{ - \la ( t+2/3)}  \Omega \( \( \begin{matrix} 1 & \la \\ 0 & Q \end{matrix} \)  \( \begin{matrix} 1 & 0 \\ 1 & 1 \end{matrix} \) z \).
\end{align}
We find that
\[
  \begin{pmatrix} 1 & \la \\ 0 & Q \end{pmatrix} \begin{pmatrix} 1 & 0 \\ 1 & 1 \end{pmatrix}  \\
= D_\la \begin{pmatrix} 1 & \la^* \\ 0 & Q/d_{\la} \end{pmatrix} \begin{pmatrix} d_{\la} & 0 \\ 0 & 1 \end{pmatrix} ,
\]
where 
\[d_{\la}: = ( 1+  \la, Q),\] 
$\la^{*}$ is chosen to satisfy the congruence
 \[\frac{1+\la}{d_{\la}} \la^{*} \equiv \la \pmod{  Q/d_{\la}},\] 
 and
\[
D_\la :=\begin{pmatrix} \frac{1 +   \la}{d_{\la}} & \frac{ - \frac{1+ \la}{d_{\la}} \la^{*} + \la}{Q/d_{\la}} \\   Q/d_{\la} & -  \la^{*} + d_{\la} \end{pmatrix}  \in \SL_2(\Z).
\]
Moreover, we may choose  the values of $\la^*$ in such a way that
\[-\la^*+d_\la\ \ \text{is even whenever $Q/d_\la$ is odd.}\]

By \eqref{OmegaTrans}, we obtain  
\begin{equation}\label{OmegaD}
\Omega \(D_\la  \frac{d_{\la} z + \la^{*}}{Q/d_{\la}} \) =
\begin{cases}
 w_1(D_\la)\cdot \( d_{\la} (  z+1) \)^{1/2}  \Omega  \( \frac{d_{\la} z + \la^{*}}{Q/d_{\la}} \) &\text{if $Q/d_{\la}$ is even,} \\
 w_2(D_\la)\cdot \( \frac{d_{\la}}{2} (  z+1) \)^{1/2}  M \( \frac{d_{\la} z + \la^{*}}{2 Q/d_{\la}} \) &\text{if $Q/d_{\la}$ is odd.}
 \end{cases} 
\end{equation}
Since $M=q^{-1/24}+\dots$ and $\Omega = q^{2/3} + \dots$, we see from \eqref{omegaqt} and \eqref{OmegaD} that the leading term in the expansion
of 
\begin{equation}\label{omegaqtexp}
(z+1)^{-1/2} \Omega_{Q,t}   \( \begin{pmatrix} 1 & 0 \\ 1 & 1 \end{pmatrix} z \)
\end{equation}
arises from those   $\lambda$ with $d_\lambda=Q$.  

The only such $\la $ is $\la_0=Q-1$. 
We may choose $\la_0^*=Q$, so that 
$D_{\la_0} :=\begin{pmatrix}  1& -1 \\ 1 & 0 \end{pmatrix}.$
Therefore the  leading term in \eqref{omegaqtexp} is 
\[\frac1{\sqrt{2Q}}  w_2 (D_{\la_0})\zeta_Q^{(1-Q)(t+2/3)} e^\frac{-2\pi iQ}{48} \cdot q^{\frac{-Q}{48}}+\dots,\]
and Proposition~\ref{cusp13} follows from \eqref{omegadefine}. 
\end{proof}

Theorem~\ref{mainthm2} follows by the argument  used to prove Theorem~\ref{mainthm1} in the last section.

\section{Proof of Theorem \ref{etathm} and Corollary \ref{etacor}}
 Recall that 
\[
\mathcal{S}(B,k,N,\chi) := \left\{ \eta^{B} (z) F(z) : \text{$F(z) \in M_{k}^{!} (\Gamma_0 (N) , \chi )$} \right\}.
\] 
Suppose that $\ell=2$ or $\ell=3$ and that 
\begin{equation}
\label{fdefinition}
f(z)= q^{B/24} \sum_{n=n_0}^\infty a_{f} (n) q^n=q^{n_0+B/24}+\dots \in \mathcal{S}(B,k,N,\chi)
\end{equation}
has rational, $\ell$-integral coefficients
and   a pole at infinity.
Given $m$ and $t$  we define
\[
f_{m,t} := \frac{1}{m} \sum_{\la=0}^{m-1} f_{\la,m,t} := \frac{1}{m} \sum_{\la=0}^{m-1} \zeta_{m}^{ - \la ( t+ B/24)}  f \( \( \begin{matrix} 1 & \la \\ 0 & m \end{matrix} \) z \). 
\] 
A  calculation as in \eqref{hmt} reveals that
\[
f_{m,t} = q^{ \frac{t+ B/24}{m}} \sum a_{f} (mn+t) q^n .
\]
We recall the transformation formula of the   eta function (see. e.g. \cite{Lewis}).
\begin{lemma} \label{etatrans}
Suppose that  $A = \(\begin{matrix} a & b \\ c & d \end{matrix} \) \in \Gamma_0 (1)$ has $c>0$. Then
\[
\eta (  A z ) = \exp \( \frac{\pi i }{12} \( \frac{ a+d}{c} - 12 s \( d, c \) \) \) \sqrt{-i (cz +d)}\cdot  \eta ( z ),
\]
where $s(d,c)$ is the Dedekind sum defined in \eqref{dedsum}.
\end{lemma}
Define
\begin{equation}\label{nmdef1}
N_m:=\begin{cases} 
m\ \ \ &\text{if}\ \ \ (m, 6)=1,\\
8m\ \ \ &\text{if}\ \ \ (m, 6)=2,\\
3m\ \ \ &\text{if}\ \ \ (m, 6)=3,\\
24m\ \ \ &\text{if}\ \ \ (m, 6)=6\\
\end{cases}
\end{equation}
(this differs slightly from the definition of $N_m$ in Section~3).
Suppose that $A=\(\begin{matrix} a & b \\ c & d \end{matrix} \) \in \Gamma_0 (N N_m  )$ has $(a, 6)=1$, and 
 let $A_\la$  and $\la^{\prime}$ be as defined in \eqref{alamdef} and \eqref{lamprimedef}.

Suppose that $k$ is an integer.
Then we have
\begin{equation} \begin{aligned}\label{halfcase}
f_{\la, m, t}  ( Az ) &= \zeta_{m}^{- \la (t + B/24)} f \( \( \begin{matrix} 1 & \la \\ 0 & m \end{matrix} \) A z \)  \\
&= \zeta_{m}^{-\la ( t + B/24 )} f \( A_{\la} \( \begin{matrix} 1 & \la^{\prime} \\ 0 & m \end{matrix} \) z\)   \\
&=\sqrt{-i (cz+d)}^{B} (cz+d)^{k} \chi (d - c \lap ) e^{\frac{B \pi i}{12} \( \frac{a+d}{mc} + \frac{\la - \lap}{m} -12s(d-c\lap,mc) \)} \\
 &\qquad\qquad\qquad\qquad\qquad\qquad\qquad\qquad\qquad\cdot \zeta_{m}^{- \la (t + B/24)} f \( \( \begin{matrix} 1 & \lap \\ 0 & m \end{matrix} \)  z \)   \\
&= \sqrt{-i (cz+d)}^{B} (cz+d)^{k} \chi (d) e^{\frac{B \pi i}{12} \( \frac{a+d}{mc} -12s(d-c\lap,mc) \)}   \zeta_{m}^{- \la t + \lap t_A } f_{\lap, m, t_A }  ,
\end{aligned}\end{equation}
where 
\begin{equation}\label{tdefnew}
t_A \equiv ta^{2} -\frac{B(1-a^2)}{24} \pmod{m}.
\end{equation}

From Lemma 2 of \cite{Lewis} (correcting the sign error), we find that
\[
s ( d - c\lap , mc) = s(d,mc) - \lap \frac{1-a^{2}}{12m} +\text{an even integer},
\]
and   that
\[
12 s (d,mc)  - \frac{a+d}{mc} + \frac{ab}{m} \in \mathbb{Z}.
\]
Since $\la \equiv a^2 \lap - ab \pmod{m}$, \eqref{halfcase} reduces to 
\[
f_{\la, m, t}  ( Az ) = \zeta_{24m}^{\Phi(A,B,m,t)} \chi (d) \sqrt{-i (cz+d)}^{B} (cz+d)^{k} f_{\lap, m, t_A},
\]
where $\Phi(A,B,m, t)$ is an integer depending only on $A$, $B$, $m$, and $t$.
We conclude that for    $A \in \Gamma_0(N N_m)$  with $(a, 6)=1$ we have 
\begin{equation} \label{etatranscon}
( f_{m,t} (A z) )^{24mN} =  (c z + d)^{24mN\(k+ B/2\)} ( f_{m,t_A} (z) )^{24mN}.
\end{equation}

If $k$ is not an integer, then the factor $\chi (d) (cz+d)^{k}$ in \eqref{halfcase}  is replaced by 
\[
\( \frac{ mc}{d-c \lap} \)\epsilon_{d-c \lap}^{-2k}  \chi (d) (cz + d)^{k}.
\]
 We have  $\epsilon_{d-c \lap } = \epsilon_{d}$ since $c$ is a multiple of $4$. 
To show that $\leg {mc}{d-c \lap}$ is independent of $\lap$,
write $mc=2^e p_1\dots p_t$  with odd primes $p_i$.  
For each $i$ we have $\leg{p_i}{d-c\lap}=(-1)^{\frac{p_i^2-1}{2}\frac{d^2-1}2}\leg{d}{p_i}$.
If $8\mid c$ then $\leg{2^e}{d-c\lap}=\leg{2^e}{d}$, while if $8\nmid c$ then $m$ is odd, so that $e=2$.
 Thus, we can  conclude \eqref{etatranscon} in this case as well.

By \eqref{etatranscon}, there is a positive integer $j$ such that
\[
f_{m,t}^{24mN} \Delta^{j} \in S_{24mN(k + B/2)+12j} ( \Gamma_{1} (N N_m) ),
\]
and if  $A=\begin{pmatrix} a & b \\ c & d \end{pmatrix}\in \Gamma_0(N N_m)$ has $(a,6)=1$ then 
\[
f_{m,t}^{24mN} \Delta^{j} \Big|  A = f_{m,t_{A}}^{24mN} \Delta^{j}.
\]
Now  recall the definition \eqref{qmb} of $Q_{m, b}$, and write $Q=Q_{m, B}$ for simplicity.
The argument in \cite[Theorem 4.2]{Radu:Crelle}
shows that as $A$ ranges over elements of $\Gamma_0(N N_m)$ with $(a,6)=1$, the quantity 
$t_{A}$ in \eqref{tdefnew}
covers each of the progressions
\[t+j Q\pmod m, \ \ \ \ 0\leq j< m/Q.\]
 So if $f_{m, t} \equiv 0 \pmod{\ell}$ we may conclude as before that $f_{Q, t}\equiv 0\pmod\ell$,
 where $(Q, N)=1$.

To obtain a contradiction we  calculate the expansion of $f_{Q,t}$ at the cusp $1/N$. 
We have
\begin{align*}
f_{Q,t} \( \( \begin{matrix} 1 & 0 \\ N & 1 \end{matrix} \) z \) &= \frac{1}{Q} \sum_{\la = 0 }^{Q-1} \zeta_{Q}^{-\la (t+B/24)} f \( \( \begin{matrix}  1 & \la \\ 0 & Q \end{matrix} \)  \( \begin{matrix} 1 & 0 \\ N & 1 \end{matrix} \) z \)  \\
&=\frac{1}{Q} \sum_{\la = 0 }^{Q-1} \zeta_{Q}^{-\la (t+B/24)} f  \( C^{\prime} \( \begin{matrix}  1 & \la^{\prime} \\ 0 & Q/d_{\la} \end{matrix} \)  \( \begin{matrix} d_{\la} & 0 \\ 0 & 1 \end{matrix} \) z \), 
\end{align*}
where $d_{\la} = \text{gcd} ( 1+ \la N , Q)$, $\lap$ is a solution of $\frac{1+\la N}{d_{\la}} \lap \equiv \la \pmod{\frac{Q}{d_{\la}}}$, and
\[
C^{\prime} = \( \begin{matrix} \frac{1+ \la N}{d_{\la}} & \frac{ -\( \frac{1+\la N}{d_{\la}} \) \lap + \la}{Q/d_{\la}} \\ NQ / d_{\la} & -N\lap + d_{\la} \end{matrix} \)\in \Gamma_0(N).
\]
 Since $f$ has a pole at   infinity, we see that the leading term of 
\[(Nz+1)^{-k - B/2} f_{Q,t}\] arises from the  unique $\la_0$ with $d_{\la_0} =Q$. Since $\lap=0$, the  coefficient of this  term is 
\[
\xi Q^{B/2 + k - 1} ,
\]
where $\xi$ is a $24QN^{\rm th}$ root of unity.   We have
 $f_{Q, t}^{24QN}\Delta^j\in S_j(\Gamma_1(NN_Q))$.
 Note that if $\ell\nmid B$ then $\ell\nmid N_Q$ by \eqref{qmb} and \eqref{nmdef1}.
Therefore  $\ell\nmid NN_Q$, so we  obtain a contradiction as before,
and 
 Theorem \ref{etathm} is proved.\qed

We finish by treating the case of eta-quotients.
 \begin{proof}[Proof of Corollary \ref{etacor}]
Suppose that $f$ is the $\eta$-quotient
 \begin{equation*}
 f(z) = \prod_{\delta \mid N } \eta ( \delta z )^{r_{\delta}}
 \end{equation*}
 and recall that 
 \begin{equation}\label{almostdone2}
 B = \sum_{\delta \mid N} \delta r_{\delta}.\end{equation}
 If $\ell\mid \delta$ for some $\delta$,
 then, writing $\delta=\ell^s\delta'$, 
 we replace the factor $\eta(\delta z)^{r_\delta}$ by the factor
 $\eta(\delta' z)^{\ell^s r_\delta}$.  These factors are congruent modulo $\ell$, 
 and the value of $B$ is not affected by this replacement.  After this discussion, 
 we may assume that $\ell\nmid N$.

 Write
  \[
 f (z) = \eta^{B} (z) \frac{ f (z)}{\eta^{B} (z)}.
 \]
 and set 
\begin{equation}\label{almostdone1}
k=\frac12\(\sum_{\delta \mid N}  r_{\delta} - B\).
\end{equation}

 Suppose first that $\ell=2$.  In this case $N$ is odd, so   $B$ and $\sum_{\delta\mid N} r_\delta$ have the same parity,
 and $k$ is an integer. We have 
 \begin{equation}\label{almostdone3}
N^2 \left( \sum_{\delta \mid N} \frac{r_{\delta}}{\delta} - B\right) \equiv 0\pmod {24}
\end{equation}
(to see this, consider the cases $3\nmid N$ and $3\mid N$ separately).
In view of  \eqref{almostdone2} and \eqref{almostdone3},  a standard criterion \cite{Newman} applies to show that $f(z)/\eta^B(z)\in M_k^!(\Gamma_0(N^2), \chi)$ for  some character $\chi$. 
If $\ell=3$ and   $k$ is an integer, then a similar argument shows that 
 $f(z)/\eta^B(z)\in M_k^!(\Gamma_0(N^4), \chi)$ for some  $\chi$.

Finally, suppose that $k$ is not an integer.  Then \eqref{almostdone1} and \eqref{almostdone2} show that 
$N$ must be even.  So we again have $f(z)/\eta^B(z)\in M_k^!(\Gamma_0(N^4), \chi)$ for  some  $\chi$.
In each of the cases,
 Corollary~\ref{etacor}  follows from 
Theorem~\ref{etathm}.
 \end{proof}


\bibliographystyle{plain}
\bibliography{fqmod3}
\end{document}